\documentclass[a4paper, 11pt]{amsart}
\usepackage[utf8]{inputenc}
\usepackage[T1]{fontenc}
\usepackage{listings}
\usepackage{textcomp}
\usepackage{upquote}
\usepackage{lmodern}
\usepackage[english]{babel}
\usepackage[dvipsnames]{xcolor}
\usepackage{adjustbox}
\usepackage{letltxmacro}
\usepackage{todonotes}
\LetLtxMacro\todonotestodo\todo
\renewcommand{\todo}[2][]{\todonotestodo[#1]{TODO: {#2}}}
\usepackage{enumerate}
\usepackage{hyperref}
\usepackage{breakurl}
\usepackage{url}
\usepackage{mathrsfs}
\usepackage{amssymb}
\usepackage{ stmaryrd }
\usepackage{braket}
\usepackage{mathrsfs}
\usepackage{tikz-cd}
\usepackage{subcaption}

\usepackage{comment}
\usepackage{float}

\theoremstyle{definition}

\newtheorem{lemma}{Lemma}[section]
\newtheorem{theorem}[lemma]{Theorem}

\newtheorem{corollary}[lemma]{Corollary}
\newtheorem{fact}[lemma]{Fact}

\newtheorem{remark}[lemma]{Remark}

\newtheorem{example}[lemma]{Example}

\newtheorem{definition}[lemma]{Definition}
\newtheorem*{claim*}{Claim}

\newtheorem*{theorem*}{Theorem}
\newtheorem*{corollary*}{Corollary}
\newtheorem*{lemma*}{Lemma}
\newtheorem*{remark*}{Remark}
\newtheorem*{question*}{Question}

\newcommand{\A}{\mathbf{A}}

\newcommand{\R}{\mathbb{R}}
\newcommand{\C}{\mathbb{C}}
\renewcommand{\P}{\mathbb{P}}

\newcommand\ddfrac[2]{\frac{\displaystyle #1}{\displaystyle #2}}

\DeclareMathOperator{\codim}{codim}

\DeclareMathOperator{\Spec}{Spec}

\DeclareMathOperator{\rlct}{rlct}
\DeclareMathOperator{\lct}{lct}

\DeclareMathOperator{\exc}{exc}
\DeclareMathOperator{\An}{An}
\renewcommand{\div}{\text{div}}
\DeclareMathOperator{\Sing}{Sing}
\DeclareMathOperator{\Jac}{Jac}

\title[Real hyperplane singularities \& real log canonical thresholds]{Classification of real hyperplane singularities \\ by real log canonical thresholds}

\author{Dimitra Kosta}
\author{Daniel Windisch}

\begin{document}
\maketitle
\thispagestyle{empty}

\begin{abstract}
The log canonical threshold (lct) is a fundamental invariant in birational geometry, essential for understanding the complexity of singularities in algebraic varieties. Its real counterpart, the \emph{real log canonical threshold} (rlct), also known as the \emph{learning coefficient}, has become increasingly relevant in statistics and machine learning, where it plays a critical role in model selection and error estimation for singular statistical models. In this paper, we investigate the rlct and its multiplicity for real (not necessarily reduced) hyperplane arrangements. We derive explicit combinatorial formulas for these invariants, generalizing earlier results that were limited to specific examples. Moreover, we provide a general algebraic theory for real log canonical thresholds, and present a SageMath implementation for efficiently computing the rlct and its multiplicity in the case or real hyperplane arrangements. Applications to examples are given, illustrating how the formulas  can also be used to analyze the asymptotic behavior of high-dimensional volume integrals.
\end{abstract}

\section{Introduction}

The log canonical threshold (lct) is a central object of study in birational geometry, see~\cite{kollar1998birational} and~\cite{lazarsfeld2004positivity}, playing a crucial role in understanding the singularities of algebraic varieties. It is an invariant that measures the complexity of singularities, and its significance spans various areas of algebraic geometry, such as the minimal model program~\cite{birkar2010existence} and the classification of singularities of algebraic varieties.

In recent years, a real analogue of the log canonical threshold, known as the \emph{real log canonical threshold} (rlct) or \emph{learning coefficient}, has become important in statistics and machine learning. Together with its \emph{multiplicity}, another birational invariant, it arises naturally in problems of error estimation and model selection, where it plays a crucial role in understanding the asymptotic behavior of statistical models. 

More specifically, the two invariants are the only missing parameters of Watanabe's criterion for model selection~\cite[Main Formula II]{Watanabe} and of its refinement due to Drton and Plummer~\cite{drton2017bayesian} known as singular Bayesian Information Criterion (sBIC). Both formulas generalize the classical Bayesian Information Criterion (BIC) to the setting of singular statistical models, such as Gaussian mixture models and neural networks, for which the BIC is invalid. In particular, the BIC is a widely used model-selection criterion that approximates the Bayesian marginal likelihood of a statistical model. It balances model fit and model complexity using the formula
$$
\text{BIC}=-2 \log L(\hat{\theta}) + d \log n ,
$$
 where $L(\hat{\theta}) $ is the maximum value of the likelihood function, $d$ is the number of free parameters in the model, and $n$ is the sample size. To use the $\text{BIC}$ for model selection, one fits each candidate model to the data, computes the BIC value for each, and then chooses the model with the lowest BIC. However, the BIC can only be used for regular models. In the case of singular models, a modification of the BIC is needed which uses real log canonical thresholds. For instance, Watanabe's approximation of the negative marginal log likelihood~\cite[Main Formula II]{Watanabe} which is given as
 $$
 -2 \log L(\hat{\theta}) + 2 \lambda \log n -2 (m-1 )\log \log n ,
 $$
  where $L(\hat{\theta}) $ is again the maximum value of the likelihood function, $n$ is the sample size, $\lambda$ is the real log canonical threshold of the true parameter set and $m$ is its multiplicity. When the model is regular, $\lambda =\frac{d}{2}$ and $m=1$, and BIC, Watanabe's formula and sBIC coincide.

  Moreover, the rlct and its multiplicity can also be used for the asymptotic evaluation of volume integrals (see \cite{Lin-dissertation, Watanabe}). In fact, the rlct and its multiplicity were initially used by Watanabe in \cite{WatanabeNeurComp2001} to prove an asymptotic approximation for the Bayesian stochastic complexity  (also known as free energy in physics). Formally, in Bayesian statistics, the Bayesian stochastic complexity  corresponds to the negative log marginal likelihood. Therefore Bayesian stochastic complexity and BIC are closely connected because BIC is an asymptotic approximation to the log marginal likelihood for regular models. In \cite{WatanabeNeurComp2001}, it was also shown that the asymptotic stochastic complexity (or free energy) is given by the Laplace transform of a volume integral. Therefore, in this way the asymptotic behaviour of volume integrals is related to the Bayesian stochastic complexity. Furthermore, the density of states, as used by Watanabe~\cite{Watanabe} in analogy to solid state physics for measuring the complexity of a statistical model around a fiber of parameters yielding the same probability distribution, is described as a volume integral, whose asymptotics is
  \[ C \varepsilon^{\lambda-1} (-\ln \varepsilon)^{m-1} \] for $\varepsilon \to 0$, see~\cite[Proof of Theorem 7.1]{Watanabe} and Section~\ref{section:examples}. 

In this paper, we investigate the rlct and its multiplicity for the class of all singularities that locally are analytically isomorphic to a real (not necessarily reduced) hyperplane arrangement. We derive explicit combinatorial formulas for these invariants which previously have been investigated mostly in very specific examples of statistical models, see the recent survey by Watanabe~\cite{watanabe2024recent}. First systematic studies are Saito's unpublished work~\cite{saito2007real}, Lin's PhD dissertation~\cite{Lin-dissertation}, and the treatment of coordinate hyperplane arrangements in the paper of Lin, Uhler, Sturmfels and Bühlmann~\cite{tubes} which we generalize.

The main contributions of this paper are:
\begin{itemize}
    \item Developing an algebraic foundation for the study of real log canonical thresholds and their multiplicities (Section~\ref{section:normal-crossing}). In particular, for a real polynomial $f \in \mathbb{R}[x_1, \ldots, x_d]$, we give a sufficient condition for the real log canonical threshold $\lambda_{\mathbb{R}}$ of $f$ to coincide with the complex log canonical threshold $\lambda_{\mathbb{C}}$ of $f$ in Lemma~\ref{lemma:lambda}, and equivalently, in Lemma~\ref{lemma:multiplicity}, a stronger sufficient condition for the corresponding multiplicities $m_{\mathbb{R}}$ and $m_{\mathbb{C}}$ to coincide. Saito~\cite{saito2007real} mentions that such results would be possible. We formalize this and give full proofs.
    \item A derivation of the complex and real log canonical threshold and its multiplicity for central hyperplane arrangements, which are products of linear forms, in Theorem~\ref{theorem:lct} and Corollary~\ref{corollary:rlct} respectively. For these results, we use the notions of a building set and wonderful compactification, which provide an appropriate log resolution for hyperplane arrangements in order to achieve explicit combinatorial formulas for the complex log canonical threshold $\lambda_{\mathbb{C}}$ in Theorem~\ref{theorem:lct} which have previously been established by  Musta{\c{t}}{\u{a}}\cite{Mustata} and Teitler\cite{Teitler08}. We show that if $f$ is a product of linear forms with real coefficients then $\text{rlct}(f)=\text{lct}(f)$, so we eventually combinatorially get the real log canonical threshold $\lambda_{\mathbb{R}}$ of a real central hyperplane arrangement in Corollary~\ref{corollary:rlct}. In combination with Remark ~\ref{remark:local-rlct}, this then gives a way to compute the real log canonical threshold for real hyperplane arrangement which need not be central.
\end{itemize}

The methods employed are grounded in the techniques of log resolutions and blowups, which facilitate resolving singularities and computing the invariants. These methods build on the approaches and results of De Concini–Procesi~\cite{DeConcini-Procesi}, Musta{\c{t}}{\u{a}}\cite{Mustata}, and Teitler\cite{Teitler08}, and they are extended to study the multiplicities of (real) log canonical thresholds, an area not yet fully explored in birational geometry.

The paper is organized as follows. In Section~\ref{section:preliminaries}, we provide the necessary preliminaries on log canonical thresholds, multiplicities, and log resolutions in both the real and complex settings. Section~\ref{section:normal-crossing} relates the notion of simple normal crossing divisors as classically used in complex birational geometry with Watanabe's notion of normal crossing functions for Singular Learning Theory.
Section~\ref{section:main} presents our main results on the rlct and its multiplicity for hyperplane arrangements, along with their proofs. In Section~\ref{section:examples}, we apply these results to several examples, illustrating how they can also be used to determine the asymptotic behavior of high-dimensional volume integrals. Finally, Section~\ref{section:code} includes a description of our linear algebra based SageMath implementation for computing rlct and multiplicities, providing practical tools for further exploration.

\section{Preliminaries}\label{section:preliminaries}

 This section is devoted to the essential background required for the statements and proofs that follow. To compute log canonical thresholds, we require the framework of resolutions of singularities, so we need to review essential background on simple normal crossing divisors, log resolutions, and blowups. Since our focus is on real log canonical thresholds of real algebraic varieties, we also provide the necessary background in real algebraic geometry that allows us to pass from the complex setting to the real one. In addition, we review the basic notions of building sets and wonderful compactifications, which in the case of hyperplane arrangements provide suitable log resolutions for the computation of real log canonical thresholds.

Let $k$ be a field and $f \in k[x_1,\ldots,x_d]$ a polynomial in $d$ indeterminates over $k$. Let $W \subseteq \A_k^d$ be a Zariski open subset. We view $f$ as a function from $W$ to $k$.

In the following sections, $k$ will either be the field $\R$ of real numbers or the field $\C$ of complex numbers. In these settings, the machinery of log resolutions described below works more generally for analytic functions $f:W \to k$, where $W \subseteq \A_k^d$ is an open subset of $\A_k^d$ (with respect to the Euclidean topology). We are interested in the case of polynomial functions in which log resolutions are a purely algebraic notion.

\subsection{$\R$-points} 
All the concepts and results of this subsection are well-known but, due to lack of a concise reference, we repeat them here.
Let $F/k$ be a field extension, where $F$ is algebraically closed. We will use the facts below in the case of the extension $\C/\R$.
We first give the general definition of a $k$-point in terms of schemes and then show that this concept aligns with our intuition of what a $k$-rational point should be in the case of an algebraic variety defined over $k$.

\begin{definition}\label{definition:k-point}
   Let $X \to \Spec(k)$ be a $k$-scheme. A \textit{$k$-rational point} (or just \textit{$k$-point}) of $X$ is a morphism of schemes $\Spec(k) \to X$ such that the following diagram is commutative:
\[
\begin{tikzcd}
\Spec(k) \arrow[r] \arrow[rd, "\text{id}"'] & X \arrow[d] \\
& \Spec(k)
\end{tikzcd}
\]
\end{definition}

In the case of $k = \R$, we will often say \textit{real point} instead of $\R$-point.

\begin{fact}\label{fact:real-max-ideals}
Let $M$ be a maximal ideal in $k[x_1,\ldots,x_d]$. Then the composition of the inclusion map $k \to k[x_1,\ldots,x_d]$ and the factor map $k[x_1,\ldots,x_d] \to k[x_1,\ldots,x_d]/M$ is a finite algebraic field extension. Moreover, the following are equivalent:
\begin{enumerate}
    \item There is a ring homomorphism $\varphi: k[x_1,\ldots,x_d] \to k$ with $\varphi|_k = \text{id}_k$ and $\ker \varphi = M$.
    \item $M = (x_1 - a_1,\ldots, x_d - a_d)$ for some $a_1,\ldots,a_d \in k$.
\end{enumerate}
\end{fact}

\begin{proof}
    The composition map $k \to k[x_1,\ldots,x_d] \to k[x_1,\ldots,x_d]/M$ is a field extension and clearly makes $k[x_1,\ldots,x_d]/M$ a finitely generated $k$-algebra. By Zariski's Lemma~\cite[Proposition 7.9]{atiyah1969commutative}, it is therefore a finite algebraic field extension.

    Statement (2) clearly implies (1). Indeed, the homomorphism $\varphi$ is just evaluation of polynomials at the point $(a_1,\ldots,a_d)$.

    So, let $\varphi$ as in (1). If we can show that for every $i \in \{1,\ldots,d\}$ there exists $a_i \in k$ such that $x_i-a_i \in M$ then the maximal ideal $(x_1- a_1,\ldots,x_d - a_d)$ is contained in the proper ideal $M$ and hence they are equal. So, assume to the contrary that there exists some $i$ such that $\{x_i - a \mid a \in k\} \cap M = \emptyset$. As $M = \ker \varphi$ and $\varphi|_k = \text{id}_k$, we derive in particular that
    \[
    \varphi(x_i) - a = \varphi(x_i) - \varphi(a) = \varphi(x_i - a) \neq 0
    \]
    for all $a \in k$. But this leads to a contradiction for $a = \varphi(x_i) \in k$.
\end{proof}

Now, we show that the intuitive notion of a $k$-point of a variety actually coincides with the abstract notion of a $k$-point.

\begin{fact}\label{fact:k-points}
    Let $k \to F$ be a field extension, where $F$ is algebraically closed.
    Let $X \to \Spec(k)$ be an integral scheme of finite type over $k$ (e.g., a variety in some projective space over $F$ locally defined by polynomials with coefficients in $k$). Let $U \subseteq X$ be an open affine subset such that $(U,\mathcal{O}_{X}|_U) \cong \Spec k[x_1,\ldots,x_d]/I$ as locally ringed spaces. (Note that $X$ can be covered by such open sets.) The following three concepts are in bijective correspondences:
    \begin{enumerate}
        \item The $k$-points $\Spec k \to X$ whose unique image point lies in $U$.
        \item The points of the variety defined by $I$ in the set-theoretical affine space $\A_F^d$ with coordinates in $k$.
        \item The maximal ideals of $k[x_1,\ldots,x_d]$ of the form $(x_1-a_1,\ldots, x_d - a_d)$ with $a_i \in k$ containing~$I$.
    \end{enumerate}
\end{fact}

\begin{proof}
    We first establish the correspondence between (1) and (3). Given a $k$-point $\Spec k \to X$ as in (1), we get an induced $k$-point $\Spec k \to U$ and a commutative diagramm
    \[
\begin{tikzcd}
\Spec(k) \arrow[r] \arrow[rd, "\text{id}"'] & U \arrow[d] \\
& \Spec(k).
\end{tikzcd}
\]
The corresponding commutative diagram of ring maps is 

\[
\begin{tikzcd}
k & k[x_1,\ldots,x_d]/I \arrow[l, "\overline{\varphi}"'] & k[x_1,\ldots,x_d] \arrow[l,"\pi"'] \\
& k \arrow[u] \arrow[lu, "\text{id}"'] \arrow[ru]
\end{tikzcd},
\]
where $\pi$ is the canonical projection map and $\overline{\varphi}$ is the ring map corresponding to $\Spec(k) \to U$.
Let $\varphi: k[x_1,\ldots,x_d] \to k$ be the composition of these two maps and let $M$ be its kernel which clearly is a maximal ideal. Moreover, $\varphi|_k = \text{id}_k$ as the diagram commutes. So, by Fact~\ref{fact:real-max-ideals}, $M$ is of the form as in (3) and contains $I$. This gives the map from (1) to (3).

To construct the inverse map from (3) to (1), let $M = (x_1- a_1,\ldots,x_d-a_d)$ with $a_i \in k$ be a maximal ideal of $k[x_1,\ldots, x_d]$ containing $I$. We again have a commutative diagram
\[
\begin{tikzcd}
k & k[x_1,\ldots,x_d]/M \arrow[l, "\cong"'] & k[x_1,\ldots,x_d]/I \arrow[l] \\
& k \arrow[u] \arrow[lu, "\text{id}"'] \arrow[ru]
\end{tikzcd},
\]
where the bottom to top maps are just the canonical inclusions and the right upper map is the canonical projection. The map $k \to k$ is indeed the identity because the above isomorphism is defined by $f(x_1,\ldots,x_d) + M \mapsto f(a_1,\ldots,a_d)$. The corresponding diagram of affine schemes gives the $k$-point $\Spec k \to \Spec k[x_1,\ldots,x_d]/I = U \subseteq X$. The map we just constructed is clearly inverse to the one above.

We now describe the correspondence between (3) and (2) which then finishes the proof. Let $M = (x_1 - a_1,\ldots, x_d - a_d)$ with $a_i \in k$ be a maximal ideal of $k[x_1,\ldots,x_d]$ containing $I$. This extends to the maximal ideal generated by the $x_i-a_i$ in $F[x_1,\ldots,x_d]$ which corresponds to the point $(a_1,\ldots,a_d) \in V(I) \subseteq \A_F^d$ by Hilbert's Nullstellensatz. This gives a map from (3) to (2). Its inverse is constructed in the following way:

Given a point $(a_1,\ldots,a_d) \in V(I) \subseteq \A_F^d$ with coordinates in $k$, by Hilbert's Nullstellensatz, we get a maximal ideal $(x_1 - a_1, \ldots, x_d - a_d)_{F[x_1,\ldots,x_d]}$ of $F[x_1,\ldots,x_d]$ containing $I$. Its contraction $M$ to $k[x_1,\ldots,x_d]$ clearly equals $(x_1 - a_1,\ldots,x_d - a_d) \subseteq k[x_1,\ldots,x_d]$ as it contains the $(x_1 - a_1,\ldots,x_d-a_d)$ which is itself a maximal ideal. Moreover, $M$ clearly contains $I$ and we are done.
\end{proof}

\subsection{Real and complex varieties}\label{real-and-complex-varieties}
For a map $g: X \to k$, where $X$ is any variety over $k$ and $X(k)$ is its set of $k$-points, we can define the \textit{zero locus} of $g$ over $k$ as 
\[
V_k(g) = \{x \in X(k) \mid g(x) = 0\}.
\]
In case $k = \R$, we call $V_\R(g)$ the \textit{real locus} of $g$, and if $k = \C$ we write $V(g) = V_\C(g)$. In all of our considerations, $X$ will locally be a $d$-dimensional affine space (over $\R$ or $\C$) and $g$ will be a polynomial. In this case, $V_\R(g)$ is simply the set of real points of the complex algebraic variety $V(g)$ defined by $g$.

Let $V$ be a complex algebraic variety and $V(\R)$ its set of real points. The following statements are equivalent by~\cite[Theorem 5.1]{Sottile} and the standard fact from general topology that a dense subset and an open subset of a topological space always intersect non-trivially; see~\cite[Theorem 2.2.9]{Mangolte} for a full proof of this equivalence in a more abstract setting:
\enlargethispage{\baselineskip}
\begin{enumerate}
    \item $V(\R)$ lies Zariski dense in $V$.
    \item Every irreducible component of $V$ contains a smooth point with real coordinates.
\end{enumerate}

\subsection{Analytic spaces}
We recall the definition of analytic spaces. Let $k$ be a field with an absolute value. For our purpose, $k$ will either be $\R$ or $\C$ with the usual Euclidean absolute value. 

Let $U$ be a subset of $k^n$ and let $f_1,\ldots,f_k$ be analytic functions on $U$. Denote by $\mathcal{A}_U$ the sheaf of analytic functions on $U$. The set $ Z = V_k(f_1,\ldots,f_m) = \{u \in U \mid \forall i \ f_i(u) = 0\}$ together with the sheaf $\mathcal{A}_Z = \mathcal{A}_U/\mathcal{I}$, where $\mathcal{I}$ is the ideal sheaf of $\mathcal{A}_U$ defined by $f_1,\ldots,f_m$, is called an \textit{analytic variety} over $k$. An \textit{analytic space} over $k$ is a locally ringed space $(X, \mathcal{A}_X)$ that can be covered by open sets that are isomorphic to analytic varieties over $k$ as locally ringed spaces. Morphisms of analytic spaces are just morphisms of locally ringed spaces.

\begin{fact}\cite[Corollary 1.6]{Artin-analytic}\label{fact:M-Artin}
   Let $X$ and $Y$ be analytic spaces, $x \in X$ and $y \in Y$. If there exists an isomorphism $\alpha: \widehat{\mathcal{A}_{X,x}} \cong \widehat{\mathcal{A}_{Y,y}}$ of completions of stalks, then there exists an isomorphism of open neighbourhoods of $x$ and $y$ inducing $\alpha$.
\end{fact}

\subsection{From varieties to analytic spaces} To each variety defined over a field $k$ with a valuation, we can associate an analytic space. We will need the analytic space structure only locally, so we restrict ourselves to affine varieties. 

Let $X = \Spec R$ be an affine variety over $k$, where $R = k[x_1,\ldots,x_n]/I$ for some ideal $I$ of $k[x_1,\ldots,x_n]$. As we saw in Fact~\ref{fact:k-points}, the $k$-points of $X$ are in one-to-one correspondence with the points in $V_k(I) = \{u \in k^n \mid (\forall f \in I) \ f(u) = 0\}$. By definition, $V_k(I)$ endowed with the Euclidean topology and its structure sheaf of analytic functions is an analytic variety. We denote this analytic variety by $\An_k(X)$ and, for simplicity, its sheaf by $\mathcal{A}_X$.

\subsection{Simple normal crossing support}\label{subsection:simple-normal-crossing}\cite[Definition 4.1.1]{LazarsfeldI}
Let $E = \sum n_i E_i$ be an effective Cartier divisor on a complex algebraic variety $X$, where the $E_i$ are prime divisors. $E$ is said to have \textit{simple normal crossing support}, or $E$ is \textit{simple normal crossing}, if the $E_i$ are regular and, for all $P \in X$, there exist Euclidean open neighbourhoods $U_P \subseteq X$ of $P$ and $U_0 \subseteq \A^d_\C$ of the origin $0 \in \A^d_\C$, and an analytic isomorphism $\alpha: U_0 \to U_P$ such that $\alpha(0) = P$ and $\alpha^* E|_{U_P}$ is defined by the function $u_1^{a_1} \cdots u_d^{a_d}$ for some non-negative integers $a_i$, where $u_i$ are the coordinate functions on $\A^d_\C$. In other words, if $E$ is defined by a regular function $g$ around $P$ then $g \circ \alpha = h \cdot u_1^{a_1} \cdots u_d^{a_d}$ in the ring $\mathcal{A}_X(U_0)$ for some unit $h$ of this ring.
Here, by an analytic isomorphism, we mean an isomorphism of the complex analytic spaces $U_0$ and $U_P$ as locally ringed spaces.

\begin{figure}[h!]
\centering

% ---- Left block: two TikZ pictures together, one subcaption ----
\begin{subfigure}[b]{0.60\textwidth}
  \centering

  \begin{minipage}{0.30\textwidth}
    \centering
    \begin{tikzpicture}[scale=1]
      \draw[thick,domain=-1.2:1.2,smooth] plot (\x,{(\x)*(\x)});
      \draw[thick] (-1.5,0) -- (1.5,0);
    \end{tikzpicture}
  \end{minipage}
  \hspace{1cm}
  \begin{minipage}{0.30\textwidth}
    \centering
    \begin{tikzpicture}[scale=1]
      \draw[thick] (-1.5,-1) -- (1.5,1);
      \draw[thick] (-1.5,1) -- (1.5,-1);
      \draw[thick] (-2,0) -- (2,0);
    \end{tikzpicture}
  \end{minipage}

  \caption{Non-simple normal crossing}
\end{subfigure}
\hspace{1cm}
% ---- Right block: third TikZ picture ----
\begin{subfigure}[b]{0.30\textwidth}
  \centering
  \begin{tikzpicture}[scale=0.9]
    \draw[thick] (-2,1) -- (2,1);
    \draw[thick] (-2,0) -- (2,0);
    \draw[thick] (-2,-1) -- (2,-1);
    \draw[thick] (-1.5,-1.5) -- (1.5,1.5);
  \end{tikzpicture}
  \caption{Simple normal crossing}
\end{subfigure}

\caption{Comparison of normal crossing configurations.}
\end{figure}

\begin{comment}
\begin{figure}[h!]
\centering

% ---- First two side-by-side, share one caption ----
\begin{minipage}{0.60\textwidth}
\centering

\begin{minipage}{0.30\textwidth}
\centering
\begin{tikzpicture}[scale=1]
\draw[thick,domain=-1.2:1.2,smooth] plot (\x,{(\x)*(\x)});
\draw[thick] (-1.5,0) -- (1.5,0);
\end{tikzpicture}
\end{minipage}
\hspace{1cm}
\begin{minipage}{0.30\textwidth}
\centering
\begin{tikzpicture}[scale=1]
\draw[thick] (-1.5,-1) -- (1.5,1);
\draw[thick] (-1.5,1) -- (1.5,-1);
\draw[thick] (-2,0) -- (2,0);
\end{tikzpicture}
\end{minipage}

\subcaption{Non-simple normal crossing}
\end{minipage}
\hspace{1cm}
% ---- Third figure ----
\begin{minipage}{0.30\textwidth}
\centering
\begin{tikzpicture}[scale=0.9]
\draw[thick] (-2,1) -- (2,1);
\draw[thick] (-2,0) -- (2,0);
\draw[thick] (-2,-1) -- (2,-1);
\draw[thick] (-1.5,-1.5) -- (1.5,1.5);
\end{tikzpicture}
\subcaption{Simple normal crossing}
\end{minipage}

\caption{Comparison of normal crossing configurations.}
\end{figure}
\end{comment}

\subsection{Morphisms defined over $\R$}\label{subsection:R-morphisms} Let $\rho: X \to Y$ be a morphism of complex varieties. By abuse of notation, we will say that \textit{$\rho$ is defined over $\R$} if there exists a morphism $\rho_\R: X_\R \to Y_\R$ of $\R$-schemes such that $\rho = \rho_\R \times_{\Spec(\R)} \Spec(\C)$. 

In other words, consider $\rho$ locally on open affine subsets of $X$ and $Y$, and let $ \varphi: B \to A$ be the corresponding homomorphism of $\C$-algebras. Then, $\rho$ is defined over $\R$ if and only if there exists an $\R$-algebra homomorphism $\varphi_\R: \R[x_1,\ldots,x_d]/I \to \R[y_1,\ldots,y_m]/J$ such that $\varphi = \varphi_\R \otimes_\R \C$, where $I$ and $J$ are ideals of the respective polynomial rings.
In particular, if there are coverings of $X$ and $Y$ by open affine subsets on which they are defined by real polynomials and on which $\rho$ is defined by regular functions with real coefficients, then $\rho$ is defined over $\R$.

\subsection{Log resolutions}
\begin{fact}\label{fact:log-resolutions}
    Let $Y$ be an irreducible complex algebraic variety, and let $D \subseteq Y$ be an effective Cartier divisor on $Y$.
    \begin{enumerate}
        \item \cite[Theorem 4.1.3]{LazarsfeldI} There is a projective birational morphism $\rho: X \to Y$, where $X$ is regular and $\rho^*D + \exc(\rho)$ is a divisor with simple normal crossing support. Here $\exc(\rho)$ denotes the exceptional set of $\rho$, namely, the set of points where $\rho$ is not a local isomorphism. A map $\rho$ satisfying this property is called a \textit{log resolution} of the pair $(Y,D)$.
        \item \cite[Theorem 4.1.3]{LazarsfeldI} One can construct $\rho$ as a sequence of blow-ups along smooth centers supported in the singular loci of $D$ and $Y$. In particular, one can assume that $\rho$ is an isomorphism over $Y \setminus (\Sing(Y) \cup \Sing(D))$.
        \item \cite[Chapter II, Theorem 4.9]{Hartshorne} Every projective morphism of varieties is proper.
        \item  If $Y = \A^d_\C$ is an affine space and $D$ is a hyperplane arrangement defined by a real polynomial then the centers of the blow-ups can additionally be chosen to be defined over $\R$. Indeed, in this case, the centers are linear spaces defined by real linear forms, see \ref{wonderful-compactifications} below. In particular, $\rho$ can be chosen to be defined over $\R$ and all prime divisors of $\rho^*D + \exc(\rho)$ are defined over $\R$.
    \end{enumerate}
\end{fact}

\subsection{$\R$-normal crossing functions}\label{subsection:R-normal-crossing} We follow~\cite[Definition 2.8]{Watanabe}. Let $X$ be a real analytic space. We say that a real analytic function $g: X \to \R$ is \textit{$\R$-normal crossing} at a point $P \in X$ if there exist (Euclidean) open neighbourhoods $U_P \subseteq X$ of $P$ and $U_0 \subseteq \A^d_\R$ of the origin $0 \in \A^d_\R$, and a real analytic isomorphism $\gamma: U_0 \to U_P$ such that $\gamma(0) = P$ and $g \circ \gamma = \eta \cdot u_1^{a_1} \cdots u_d^{a_d}$ for some non-negative integers $a_i$, where $u_i$ are the coordinate functions on $\A^d_\R$ and $\eta$ is a real analytic function that has no zeros on $U_0$.

\subsection{Blow ups} Log resolutions exist over every field $k$ of characteristic zero by a celebrated result of Hironaka~\cite{hironaka1964resolution} and can be constructed using a composition of so-called blow ups along proper subvarieties of $V_k(f)$. In the particular case of affine $d$-space over $k$, the \textit{blow up} along the scheme defined by $f_1,\ldots,f_\nu \in k[x_1,\ldots,x_d]$ is the projection map
$\pi_{f_1,\ldots,f_\nu}: \text{Bl}_{f_1,\ldots,f_\nu}(\A_k^d) \to \A_k^d$, where
\[
 \text{Bl}_{f_1,\ldots,f_\nu}(\A_k^d) = \overline{\{ (x,(f_1(x): \ldots : f_\nu(x))) \in \A_k^d \times \P^{\nu-1}_k \mid x \in \A_k^d \setminus V_k(f) \}}.
\]
Often, the $d$-dimensional variety $\text{Bl}_{f_1,\ldots,f_\nu}(\A_k^d)$ is itself called the blow up. Intuitively speaking, to construct $\text{Bl}_{f_1,\ldots,f_\nu}(\A_k^d)$, every point $x$ in the variety $V$ defined by the $f_1,\ldots,f_\nu$ is replaced by the projective space of lines through $x$ perpenticular to $V$. Moreover, the blow up is an isomorphism outside of the scheme defined by $f_1,\ldots,f_\nu$.

Let $S \subseteq \A_k^d$ be a closed subscheme not contained in the closed subscheme defined by $f_1,\ldots,f_\nu$. The closure of the pullback under $\pi_{f_1,\ldots,f_\nu}$ of the open subscheme defined in $S$ by $f_1,\ldots,f_\nu$ is called the \textit{proper transform} of $S$ under $\pi_{f_1,\ldots,f_\nu}$. The pullback of the closed subscheme defined by $f_1,\ldots,f_\nu$ is called the \textit{exceptional divisor} of the blowup.

\subsection{Wonderful compactifications}\label{wonderful-compactifications} Let $k$ be a field. By a \textit{hyperplane arrangement} in $\A^d_k$, we mean the closed subscheme defined by a function $f$ that is the finite product of linear polynomials in $d$ variables with coefficients in $k$. Note that this subscheme is not necessarily reduced. For hyperplane arrangements over $\C$, De Concini and Procesi~\cite{DeConcini-Procesi} discovered a uniform way of constructing log resolutions by blow ups. This applies more generally to schemes that locally look like a union of hyperplanes~\cite{wonderful-compactifications}.

Let $f = \prod L_i^{s_i}$ define a hyperplane arrangement and $H_i$ the hyperplane defined by $L_i$. Denote by $\mathcal{A} = \{H_1,\ldots, H_n\}$ the (set-theoretical) hyperplane arrangement defined by $f$, and let $\overline{L}(\mathcal{A})$ be the intersection lattice of $\mathcal{A}$, that is, the set of all intersections of the $H_i$. By $L(\mathcal{A})$ we denote the set of all elements of $\overline{L}(\mathcal{A})$ except for the ambient affine space.

%Every element $W \in L(f)$ gives rise to two numbers: $r(W)$ is its codimension in $\R^d$ and $s(W) = \sum_{W \subseteq H_i \in \mathcal{A}} s_i$.

%A \textit{decomposition} of $C \in L(f)$ is a subset $\{U_1,\ldots,U_k\} \subseteq L(\mathcal{A})$ such that 
%\begin{itemize}
    %\item[(i)] $C = U_1 \cap \ldots \cap U_k$ transversally, that is, $\codim C = \codim U_1 + \cdots + \codim U_k$, and
    %\item[(ii)] for every $C \subseteq B \in L(\mathcal{A})$, the sum of vector subspaces $B + U_i$ is in $\overline{L}(\mathcal{A})$, for every $i$, and $B = (B+U_1) \cap \ldots \cap (B+U_k)$ transversally.
%\end{itemize}
A subset $\mathcal{G} \subseteq L(\mathcal{A})$ is a \textit{building set} for $\mathcal{A}$ if, for every $C \in L(\mathcal{A})$, the set of those elements of $\mathcal{G}$ minimally containing $C$ intersect transversally and the intersection is $C$. So, the set $L(\mathcal{A})$ itself is a building set.

%Every such $C$ admits the trivial decomposition $\{ C \}$ and a $C$ only having this trivial decomposition is called \textit{irreducible}. Clearly, every building set has to contain all irreducible $C$. De Concini and Procesi~\cite{DeConcini-Procesi} showed that the set $\mathcal{G}_{\min}$ of all irreducible elements of $L(\mathcal{A})$ forms a building set. 
Take any total order $\leq$ on $\mathcal{G}$ such that $V\subseteq W$ implies $V \leq W$ for all $V,W \in \mathcal{G}$. The \textit{wonderful compactification} with respect to $\mathcal{G} = \{U_1 \leq U_2 \leq \ldots \leq U_N\}$ is a sequence of blow-ups, first blowing up along $U_1$, then blowing up the result at the proper transform of $U_2$, and so on. It only depends on the order $\leq$ up to isomorphism~\cite[Proposition 2.13]{wonderful-compactifications} and it is a log resolution.

It follows from the construction that the prime divisors of the pullback of $f$ under the wonderful compactification with respect to a building set $\mathcal{G}$ are in bijection with the elements of $\mathcal{G}$, see~\cite{Teitler08} for a full treatment. This bijection is given in the following way: Let $i \in \{1,\ldots,N\}$ so that $U_i \in \mathcal{G}$. We inductively define an closed integral subscheme $D_{U_i}^{(j)}$ in the $j$-th blow up step of the wonderful compactification, for $j \in \{0,\ldots,N\}$. We set $D_{U_i}^{(0)} = U_i$. For $j \in \{0,\ldots,N\}$, let $D_{U_i}^{(j+1)}$ be the proper transform of $D_{U_i}^{(j)}$ in the $(j+1)$-st blow up step if $i \neq j+1$, and let $D_{U_i}^{(j+1)}$ be the exceptional divisor of the $(j+1)$-st blow up step if $i = j+1$. Now, as we took an exceptional divisor of a blow up exactly once, and a proper transform in all other cases, $D_{U_i} = D_{U_i}^{(N)}$ is actually a prime divisor of the pullback of $f$, and all prime divisors of $f$ are of this form.

\section{Real and complex normal crossing divisors}\label{section:normal-crossing}
 In Section~\ref{section:normal-crossing}, we will explore the relationship between two notions of normal crossing divisors: simple normal crossing divisors (see~\ref{subsection:simple-normal-crossing}), commonly encountered in complex algebraic geometry, and $\R$-normal crossing functions (see~\ref{subsection:R-normal-crossing}), as introduced in Watanabe's book on Singular Learning Theory~\cite{Watanabe}. While simple normal crossing divisors rely on arbitrary local analytic isomorphisms over $\C$, the notion of $\R$-normal crossing is constrained to those defined over $\R$. Theorem~\ref{theorem:log-resolutions} establishes that, in a specific sense, these two notions align. This alignment is a foundational step toward computing real log canonical thresholds based on complex log canonical thresholds.

\subsection{Log resolutions and $\R$-normal crossing}
In order to use results from complex birational geometry for our study of real log canonical thresholds, we want to relate the notions of a log resolution with the concept that, in order to avoid confusion, we called $\R$-normal crossing. It might seem trivial that a simple normal crossing divisor defines an $\R$-normal crossing function at every $\R$-point of the ambient scheme, as both definitions use a local analytic isomorphism $\alpha$. However, for $\R$-normal crossing, $\alpha$ needs to be an isomorphism of analytic spaces over $\R$. This is the subtle part of the following result which we achieve by passing to maps of completions of local rings. 
All results of this section apply to real hyperplane arrangements by Fact~\ref{fact:log-resolutions}(4).

%We only use that $f$ defines a hyperplane arrangement to make sure that there exists a log resolution that is a sequence of blow-ups with centers defined over $\R$, see Fact~\ref{fact:log-resolutions}(4). Hence, the results of this section apply to any $f$ with this property.

\begin{theorem}\label{theorem:log-resolutions}
    Let $f \in \R[x_1,\ldots,x_d]$ be a real polynomial such that there exists a log resolution $\rho: X \to \A^d_\C$ for $(\A^d_\C,\div(f))$ that is defined over $\R$. Then $f \circ \rho$ is $\R$-normal crossing at every $\R$-point $P$ of $X$.
\end{theorem}

\begin{proof}
    By assumption, we find a log resolution $\rho: X \to \A_\C^d$ defined over $\R$, and all the prime divisors of $\rho^*\div(f) + \exc(\rho) = \div(f \circ \rho)$ are defined over $\R$. Let $\rho_\R: X_\R \to \A_\R^d$ be such that $\rho = \rho_\R \times_{\Spec(\R)} \Spec(\C)$, see~\ref{subsection:R-morphisms}. Let $P$ be an $\R$-point of $X$. As $\rho$ is a log resolution, we know that $f \circ \rho$ is simple normal crossing at $P$. So, let $U_P \subseteq X$ and $U_0 \subseteq \A_\C^d$ be Euclidean open neighbourhoods of $P$ and the origin, respectively, and let $\alpha: U_0 \to U_P$ be an isomorphism of analytic spaces over $\C$ such that $\alpha(0) = P$ and 
    \begin{equation}\label{equation:at-zero}\tag{$\ast$}
        f \circ \rho \circ \alpha = h \cdot u_1^{a_1} \cdots u_d^{a_d} 
    \end{equation}
    in the ring $\mathcal{A}_{\A_\C^d}(U_0)$ for some unit $h$ of this ring.

Both the ring of germs of regular functions in a point of $d$-dimensional affine space (over $\R$ and $\C$) and the ring of germs of analytic functions in the same point complete to the ring of formal power series in $d$ variables (see \cite{cohen}). Moreover, completion and factoring out an ideal commute.
Therefore, the inclusions $\mathcal{O}_{X_\R,P} \hookrightarrow \mathcal{A}_{X_\R,P}$ and $\mathcal{O}_{X,P} \hookrightarrow \mathcal{A}_{X,P}$ induce isomorphisms of completions of these rings. 
    
    By the canonical map $\mathcal{A}_X(U_P) \to \widehat{\mathcal{A}_{X,P}}$, we can consider $f \circ \rho$ as an element of $\widehat{\mathcal{A}_{X,P}}$. The map $\alpha$ induces a $\C$-algebra isomorphism $\widehat{\alpha}: \widehat{\mathcal{A}_{X,P}} \to \widehat{\mathcal{A}_{\A_\C^d,0}}$. Moreover, by the previous paragraph, the ring $\widehat{\mathcal{A}_{\A_\C^d,0}}$ is a complete regular local ring and $u_1,\ldots,u_d$ form a regular system of parameters. Cohen's Structure Theorem yields an isomorphism of $\C$-algebras $\beta: \widehat{\mathcal{A}_{\A_\C^d,0}} \to \C\llbracket u_1,\ldots,u_d \rrbracket$ with $u_i \mapsto u_i$.

    Let $g_i := \widehat{\alpha}^{-1} \circ \beta^{-1}(u_i) \in \widehat{\mathcal{A}_{X,P}}$. Then, the $g_i$ with $a_i \neq 0$ locally at $P$ define the prime divisors of $\div(f \circ \rho)$ and are, hence, defined over $\R$. Also, the $g_i$ are a regular system of parameters in $\widehat{\mathcal{A}_{X,P}}$ as $\beta \circ \widehat{\alpha}$ is an isomorphism. 

    %$X$ is regular and irreducible and contains the $\R$-point $P$, which is therefore a smooth point. It follows that the dimension of $\An_\R(X_\R)$ equals $d$ and that the Krull dimension of $\widehat{\mathcal{A}_{X_\R,P}}$ is also $d$. 

    As $\rho = \rho_\R \times_{\Spec(\R)} \Spec(\C)$, we have in particular that $X = X_\R \times_{\Spec(\R)} \Spec(\C)$.
    Therefore, we have a canonical isomorphism  and two canonical inclusions:
    \[
\begin{tikzcd}
\mathcal{O}_{X_\R,P} \otimes_\R \C \arrow[r, "\cong"] \arrow[d, hookrightarrow] & \mathcal{O}_{X,P} \arrow[d, hookrightarrow] \\
\mathcal{A}_{X_\R,P} \otimes_\R \C  & \mathcal{A}_{X,P}
\end{tikzcd}
\]
 Moreover, $\widehat{\mathcal{O}_{X_\R,P}} \otimes_\R \C$ is a complete ring because $\C$ is a finite field extension of $\R$. Using this and completing the above diagram, we get a commutative diagram

\[
\begin{tikzcd}
\widehat{\mathcal{O}_{X_\R,P}} \otimes_\R \C \arrow[r, "\cong"] \arrow[d, "\cong"] & \mathcal{O}_{X_\R,P} \widehat{\otimes}_\R \C\arrow[r, "\cong"] \arrow[d, hookrightarrow] & \widehat{\mathcal{O}_{X,P}} \arrow[d, "\cong"] \\
\widehat{\mathcal{A}_{X_\R,P}} \otimes_\R \C \arrow[r, "\cong"] & \mathcal{A}_{X_\R,P} \widehat{\otimes}_\R \C  & \widehat{\mathcal{A}_{X,P}},
\end{tikzcd}
\]
where $\widehat{\otimes}$ means first taking tensor product and then completion.
So, we have an isomorphism $\widehat{\mathcal{A}_{X_\R,P}} \otimes_\R \C \cong \widehat{\mathcal{A}_{X,P}}$ respecting functions and, in particular, mapping $g_i \mapsto g_i$. Hence, the $g_i$ form a regular system of parameters in $\widehat{\mathcal{A}_{X_\R,P}} \otimes_\R \C$. Moreover, as $\C$ is a faithfully flat $\R$-module and tensoring by a faithfully flat module preserves regular systems of parameters, the $g_i$ also form a regular system of parameters in $\widehat{\mathcal{A}_{X_\R,P}}$. Cohen's Structure Theorem yields an isomorphism $\widehat{\gamma}: \widehat{\mathcal{A}_{X_\R,P}} \cong \R\llbracket u_1,\ldots, u_d\rrbracket \cong \widehat{\mathcal{A}_{\A_\R^d,0}}$ mapping $g_i$ to $u_i$. By Fact~\ref{fact:M-Artin}, this isomorphism must come from an analytic isomorphism $\gamma$ of real Euclidean open neighbourhoods $V_0$ of $0$ and $V_P$ of $P$, respectively.

Using the isomorphism $\widehat{\mathcal{A}_{X_\R,P}} \otimes_\R \C \cong \widehat{\mathcal{A}_{X,P}}$, the fact that $f \circ \rho$ is defined over $\R$ and~(\ref{equation:at-zero}) above, we see that $f\circ \rho = \widehat{\eta} \cdot g_1^{a_1} \cdots g_d^{a_d}$ as elements of $\widehat{\mathcal{A}_{X_\R,P}}$, where $\widehat{\eta}$ is a unit of this ring. Therefore, $\widehat{\gamma}(\widehat{\eta})$ is a unit in $\R\llbracket u_1,\ldots,u_d\rrbracket$, that is, a (convergent) power series with non-zero constant term. So, $\widehat{\gamma}(\widehat{\eta})$ is the germ of a function that has no zeros in some open Euclidean neighbourhood $U$ of the origin. We can hence replace $V_P$ by the smaller open neighbourhood $V_P' = V_P \cap U$ of $P$ and choose a function $\eta \in \mathcal{A}_{X_\R}(V_P')$ inducing $\widehat{\eta}$ such that $\eta \circ \gamma$ has no zeros on $V_0' := \gamma^{-1}(V_P')$. Hence, the equality $f \circ \rho \circ \gamma = (\eta \circ \gamma) \cdot u_1^{a_1} \cdots u_d^{a_d}$ shows that $f \circ \rho$ is $\R$-normal crossing at $P$.
\end{proof}

\subsection{The relative canonical divisor}\label{subsection:relative-canonical-divisor}
Let $D$ be an effective Cartier divisor on $\A^d_\C$. Let $\rho: X \to \A^d_\C$ be a log resolution of the pair $(\A^d_\C,D)$ and let $\rho^*D + \exc(\rho) = \sum_{i = 1}^m a_i E_i$, where the $E_i$ are prime divisors on $X$ and the $a_i$ are positive integers.

 For simplicity, we use a local description of the \textit{relative canonical divisor} $K_\rho$ of $\rho$, see~\cite[Remark 4.2]{Popa-lct-notes}. Let $P \in X$ and $f$ be a function defining $D$ locally at $P$. Let $\alpha$ be a local analytic isomorphism such that $f \circ \rho \circ \alpha = h \cdot u_1^{a_{i_1}} \cdots u_k^{a_{i_k}}$, with $k \in \{0,\ldots,d\}$ and $i_j \in \{1,\ldots,m\}$, as in Section~\ref{subsection:simple-normal-crossing} above. The relative canonical divisor of $\rho$ is defined locally at $P$ by the determinant $\det \Jac (\rho \circ \alpha)$ of the Jacobian matrix of $\rho \circ \alpha$ with respect to the local coordinates $u_1,\ldots,u_d$.

 It is a standard fact in birational geometry that the prime divisors of $K_\rho$ are among the $E_i$, so that $K_\rho = \sum_{i = 1}^m b_i E_i$ for some non-negative integers $b_i$. In particular, $\det \Jac(\rho \circ \alpha) = h' \cdot u_1^{b_{i_1}} \cdots u_d^{b_{i_d}}$, for some section $h'$ of $\mathcal{A}_X$ that is locally a unit at $P$.

 If $P$ is an $\R$-point of $X$ then, following the proof of Theorem~\ref{theorem:log-resolutions}, we get the local representation $\det \Jac(\rho \circ \gamma) = \eta' \cdot u_1^{a_{i_1}} \cdots u_d^{a_{i_d}}$, where $\gamma$ is a local analytic isomorphism from a neighbourhood of $P$ to a neighbourhood of $0 \in \A^d_\R$ and $\eta'$ has no zeros around $0 \in \A^d_\R$, see also Section~\ref{subsection:R-normal-crossing}.

\subsection{The log canonical threshold}\label{subsection:log-canonical-threshold} Let $f \in \C[x_1,\ldots,x_d]$ and let $\rho: X \to \A^d_\C$ be a log resolution for $(\A^d_\C,\div(f))$. Write
\[
\div(f \circ \rho) = \sum_{i = 1}^m a_i E_i
\]
and 
\[
K_\rho = \sum_{i = 1}^m b_i E_i,
\]
where $a_i$ are positive integers and $b_i$ are non-negative integers.

\begin{definition}
  The \textit{log canonical threshold} $\lambda_\C$ of $f$ is given by
\[
\lambda_\C = \min_{i \in \{1,\ldots,m\}} \frac{b_i + 1}{a_i}.
\]  
\end{definition}

To $f$ and a compact subset $W \subseteq \C^d$, we can associate a zeta function $\zeta_f: \C \to \C$ defined by $\zeta_f(s) = \int_{W} |f(w)|^{2s} dw$. All the poles of these functions are real numbers and the largest among the poles is~$-\lambda_\C$.

\begin{definition}
    We call  the \textit{multiplicity} of the log canonical threshold of $f$ and denote it by $m_\C$, the maximum order of $-\lambda_\C$ as pole of a zeta function associated to $f$, which is given by
    \[
    m_\C =  \max_{P \in X} \bigg\vert \Set{  i \in \{1,\ldots,m\} | \lambda_\C = \frac{b_i + 1}{a_i}, P \in E_i } \bigg\vert.
    \]
\end{definition}

By the interpretation via zeta functions, we see that the pair $\lct(f) = (\lambda_\C,m_\C)$ is independent from the choice of log resolution~$\rho$.

\subsection{The real log canonical threshold} Let $f \in \R[x_1,\ldots,x_d]$. We follow~\cite[Definition~2.7 \& p. 34]{Watanabe} for the definition of the real log canonical threshold of $f$ and its multiplicity; see also~\cite[Theorem 7.1]{tubes}. Let $U$ be a $d$-dimensional real manifold and $\sigma: U \to \A^d_\R$ a proper real analytic map that restricts to an isomorphism $ U \setminus \sigma^{-1}(V_\R(f)) \to \A^d_\R \setminus V_\R(f)$ and satisfies the following property: For every $P \in U$, there exists an analytic isomorphism $\gamma: U_0 \to U_P$, where $U_0$ and $U_P$ are open Euclidean neighbourhoods of $0 \in \A^d_\R$ and $P \in U$, respectively, such that 
\[
f \circ \sigma \circ \gamma = \eta \cdot u_1^{a_1} \cdots u_d^{a_d}
\]
and
\[
\det \Jac(\sigma \circ \gamma) = \eta' \cdot u_1^{b_1} \cdots u_d^{b_d}
\]
for some non-negative integers $a_i, b_i$ and some function $\eta,\eta' \in \mathcal{A}_U(U_P)$ that have no zeros on $U_P$.

\begin{definition}
The \textit{real log canonical threshold} $\lambda_\R$ of $f$ and its \textit{multiplicity} $m_\R$ are defined by 
\[
\lambda_\R =  \inf_{\substack{P \in U \\ i \in \{1,\ldots,d\}, a_i \neq 0}} \frac{b_i + 1}{a_i} 
\]
and
\[
m_\R = \max_{P \in U} \bigg\vert \Set{  i \in \{1,\ldots,d\} | a_i \neq 0, \lambda_\R = \frac{b_i + 1}{a_i}} \bigg\vert.
\]
We define $\rlct(f) = (\lambda_\R,m_\R)$.
\end{definition}

 This pair is independent from the choice of $\sigma$ by an interpretation similar to the one in Section~\ref{subsection:log-canonical-threshold} of $-\lambda_\R$ as largest pole of zeta functions associated to $f$; see also~\cite[Theorem 2.4]{Watanabe}. So, we might choose $U = X_\R$ and $\sigma = \rho_\R$, where $\rho: X \to \A_\C^d$ is a log resolution as in Theorem~\ref{theorem:log-resolutions}. As in the discussion of Section~\ref{subsection:relative-canonical-divisor}, we write
\[
\div(f \circ \rho) = \sum_{i = 1}^m a_i E_i \text{ and } K_\rho = \sum_{i = 1}^m b_i E_i,
\]
where the $a_i$ are positive integers and the $b_i$ are non-negative integers. Using this, we get

\[
    \lambda_\R = \min_{\substack{i \in \{1,\ldots,m\} \\ E_i \cap X(\R) \neq \emptyset}} \frac{b_i + 1}{a_i}
\]
and
\[
    m_\R =  \max_{P \in X(\R) } \bigg\vert \Set{  i \in \{1,\ldots,m\} | \lambda_\R = \frac{b_i + 1}{a_i}, P \in E_i } \bigg\vert.
\]

\subsection{$\lambda_\R$ and $m_\R$ from a log resolution}
The following lemma is a sufficient condition of when $\lambda_\R = \lambda_\C$ for a polynomial $f \in \R[x_1,\ldots,x_d]$.

\begin{lemma}\label{lemma:lambda}
    Let $f \in \R[x_1,\ldots,x_d]$ and $\rho$ a log resolution as in Theorem~\ref{theorem:log-resolutions}. We use the notation from Section~\ref{subsection:relative-canonical-divisor}. Suppose that every prime divisor $E_i$ of $\div(f \circ \rho)$, for $i \in \{1,\ldots,m\}$, contains an $\R$-point of $X$. Then $\lambda_\R = \lambda_\C$.
\end{lemma}

\begin{proof}
    By definition
    \[
    \lambda_\C = \min_{i \in \{1,\ldots,m\}} \frac{b_i + 1}{a_i}.
    \]
    Using Theorem~\ref{theorem:log-resolutions} and the discussion in Section~\ref{subsection:relative-canonical-divisor}, we get that
    \[
    \lambda_\R = \min_{\substack{i \in \{1,\ldots,m\} \\ E_i \cap X(\R) \neq \emptyset}} \frac{b_i + 1}{a_i},
    \]
    where $X(\R)$ is the set of $\R$-points of $X$, and the statement follows.
\end{proof}

We need a stronger condition to also infer the equality $m_\C = m_\R$.

\begin{lemma}\label{lemma:multiplicity}
    Let $f \in \R[x_1,\ldots,x_d]$ and $\rho$ a log resolution as in Theorem~\ref{theorem:log-resolutions}. We use the notation from Section~\ref{subsection:relative-canonical-divisor}. Suppose that for every $P \in X$ there exists an $\R$-point $Q \in X$ such that
    \[
        \{i \mid P \in E_i\} = \{i \mid Q \in E_i\}.
    \]
    Then $m_\R = m_\C$.
\end{lemma}

\begin{proof}
By definition
    \[
    m_\C =  \max_{P \in X} \bigg\vert \Set{  i \in \{1,\ldots,m\} | \lambda_\C = \frac{b_i + 1}{a_i}, P \in E_i } \bigg\vert.
    \]
Moreover, the assumption of the lemma implies, in particular, that every $E_i$ contains an $\R$-point of $X$. So, by Lemma~\ref{lemma:lambda}, $\lambda_\R = \lambda_\C$. Using this equality,  Theorem~\ref{theorem:log-resolutions} and the discussion in Section~\ref{subsection:relative-canonical-divisor}, we get
    \[
    m_\R =  \max_{Q \in X(\R) } \bigg\vert \Set{  i \in \{1,\ldots,m\} | \lambda_\C = \frac{b_i + 1}{a_i}, Q \in E_i } \bigg\vert.
    \]
The assumption of the lemma implies that, for each $P \in X$, there exists $Q \in X(\R)$ such that 
\[
\Set{  i \in \{1,\ldots,m\} | \lambda_\C = \frac{b_i + 1}{a_i}, P \in E_i }  = \Set{  i \in \{1,\ldots,m\} | \lambda_\C = \frac{b_i + 1}{a_i}, Q \in E_i },
\]
and the statement of the lemma follows.
\end{proof}

\section{Real log canonical thresholds of hyperplane arrangements}\label{section:main}

In this section, we will restrict ourselves to central hyperplane arrangements, that is, closed subschemes of an affine space defined by a product of linear forms rather than just any product of linear polynomials. However, this is no restriction, as we will see in Remark~\ref{remark:local-rlct}.

Before applying results from birational geometry concerning wonderful compactifications $\rho$ of hyperplane arrangements $f$ to our study of the multiplicities of (real) log canonical thresholds, we must first determine which prime divisors in the pullback of $f$ under $\rho$ are supported at common points. The following lemma provides a characterization in the context of the maximal building set $L(\mathcal{A})$.

\begin{lemma}\label{lemma:intersection-in-blowup}
    Let $f \in \C[x_1,\ldots,x_d]$ be a product of linear forms and $\mathcal{A}$ be the set-theoretical hyperplane arrangement defined by $f$ (see~\ref{wonderful-compactifications}).
    Let $W_0,\ldots,W_r \in L(\mathcal{A})$ be pairwise distinct and let $D_{W_0},\ldots,D_{W_r}$ be their corresponding prime divisors in the wonderful compactification of $\mathcal{A}$ with respect to the building set $L(\mathcal{A})$. Then, the following are equivalent:
    \begin{enumerate}
        \item $\bigcap_{i = 0}^r D_{W_i} \neq \emptyset$.
        \item There exists a permutation $\sigma$ of $\{0,\ldots,r\}$ such that $W_{\sigma(0)} \subseteq W_{\sigma(1)} \subseteq \ldots \subseteq W_{\sigma(r)}$.
    \end{enumerate}
    In particular, if (1) holds then $r \leq d-1$. Moreover, if $r = d-1$ then $\dim_\C W_{\sigma(i)} = i$ for all $i \in \{0,\ldots,d-1\}$.
\end{lemma}

\begin{proof}
    If (2) does not hold then there exist $i,j \in \{0,\ldots,r\}$ such that neither $W_i \not\subseteq W_j$ nor $W_j \not\subseteq W_i$. Consequently, $V = W_i \cap W_j$ is a transversal intersection and the exceptional divisors of $W_i$ and $W_j$ are separated after the blow up along $V$. In particular, (1) does not hold.

    Suppose that (2) holds and assume without loss of generality that $r = d-1$ and that $W_0 \subseteq W_1 \subseteq \ldots \subseteq W_{d-1}$. Write $L(\mathcal{A}) = \{U_1 \leq \ldots \leq U_N\}$, where $\leq$ is a total order on $L(\mathcal{A})$ compatible with set-theoretical inclusion. We do induction on $j \in \{1,\ldots,N\}$, showing that in each step of the sequence of blow ups along the $U_j$, the proper transforms respectively the exceptional divisors given by the $W_i$ have non-empty intersection.

    As basis of the induction, the initial intersection $\bigcap_{i = 1}^r W_i$ is non-empty because the $W_i$ are linear subspaces of $\C^d$. Let $j \in \{1,\ldots,N-1\}$ and assume that, after the blow up along (the proper transform of) $U_j$, the proper transforms respectively exceptional divisors of the $W_i$ have non-empty intersection. 
    Also, suppose that $W_0,\ldots, W_k$ are among the $U_1,\ldots,U_j$, and $W_{k+1},\ldots,W_{d-1}$ are not.

    Then $W_k \cap U_{j+1} \subsetneqq U_{j+1}$ and, hence, the blow up along this intersection has already been performed before the $(j+1)$-st blow up step. In particular, the divisors induced by $W_k$ and $U_{j+1}$, respectively, after the $j$-th blow up are disjoint, unless $W_k \subseteq U_{j+1}$. If they are indeed disjoint, then so is the intersection of the induced divisors of the $W_i$, and so an intersection point outside the proper transform of $U_{j+1}$ is preserved also after the step $j+1$. If $W_k \subseteq U_{j+1}$ then $\dim U_{j+1} = k+1$ and there are two cases. 

    The first case arises when $U_{j+1} \not\subseteq W_{k+1}$. Then $U_{j+1} \cap W_{k+1} \subsetneqq U_{j+1}$ and the same argument as in the previous paragraph applies by replacing $W_k$ with $W_{k+1}$. In the second case, when $U_{j+1} \subseteq W_{k+1}$, we deduce $U_{j+1} = W_{k+1}$ for dimension reasons. Clearly, every intersection point is preserved in this case. 
\end{proof}

We are now prepared to present a combinatorial formula for the pair $\lct(f)$ associated with a (not necessarily reduced) hyperplane arrangement $f$ over $\C$. Teitler previously provided the description for $\lambda_\C$ in~\cite{Teitler08}, and Mustaţă~\cite{Mustata} demonstrated it earlier in the case of a reduced hyperplane arrangement.

\begin{theorem}\label{theorem:lct}
    Let $f = L_1^{s_1} \cdots L_n^{s_n}$, where each $L_i \in \C[x_1,\ldots,x_d]$ is a linear form and the $s_i$ are positive integers. Let $\mathcal{G}$ be a building set of the hyperplane arrangement $\mathcal{A}$ defined by $f$ and let $\lct(f) = (\lambda_\C,m_\C)$ be the log canonical threshold of $f$ and its multiplicity. Then 
    \[
    \lambda_\C = \min_{W \in \mathcal{G}} \ddfrac{\text{codim}(W)}{s(W)}  \text{ and } m_\C = \max_{{ \substack{W_0\subseteq \ldots \subseteq W_{d-1}\\ W_i \in L(\mathcal{A}) } }} \bigg\vert \Set{  i \in \{0,\ldots,d-1\} | \lambda= \frac{\text{codim}(W_i)}{s(W_i)}} \bigg\vert,
    \]
    where, for $W \in L(\mathcal{A})$,
    \[
    s(W) = \sum_{{\substack{j = 1 \\ W \subseteq H_j}}}^n s_j.
    \]
\end{theorem}

\begin{proof}
    Using the notation from~\ref{wonderful-compactifications}, and using~\cite[Lemma 2.1]{Teitler08}, the pullback divisor of $f$ under the wonderful compactification $\rho:U \to \C^d$ with respect to $\mathcal{G}$ and the divisor defined by $\det \text{Jac}(\rho)$ are
    \[
    \sum_{W \in \mathcal{G}} s(W) D_W
    \]
    and
    \[
    \sum_{W \in \mathcal{G}} (\codim(W)-1)D_W,
    \]
    respectively. So the formula for $\lambda_\C$ follows by definition of the log canonical threshold.

    Let $P \in U$. Then the relative canonical divisor and the pullback locally at $P$ are of the form
    \[
    \sum_{{ \substack{W \in \mathcal G \\ P \in D_W}}} (\codim(W)-1)D_W \text{ and } \sum_{{ \substack{W \in \mathcal G \\ P \in D_W}}} s(W) D_W.
    \]
    If $\mathcal{G} = L(\mathcal{A})$ then, by Lemma~\ref{lemma:intersection-in-blowup}, the maximal sets of elements $W \in L(\mathcal{A})$ whose induced divisors $D_W$ intersect in a point $P \in U$ are the sets $\{W_0 \subseteq \ldots \subseteq W_{d-1}\}$. So the formula for $m_\C$ follows by definition of the multiplicity.
\end{proof}

We saw in Lemmas~\ref{lemma:lambda} and~\ref{lemma:multiplicity} that the only obstacle for the equality $\rlct(f) = \lct(f)$ is that some irreducible components of $f \circ \rho$, where $\rho$ is a log resolution, might have common complex points but no common real points. We will now show that this is not true when $f$ defines a real hyperplane arrangement.

\begin{theorem}\label{theorem:lct=rlct_for_hyperplane}
    Let $f \in \R[x_1,\ldots,x_d]$ be a product of linear forms with real coefficients. Then $\rlct(f) = \lct(f)$.
\end{theorem}

\begin{proof}
    Let $\rho: X \to \A^d_\C$ be the wonderful compactification with respect to the building set $\mathcal{L}$ that consists of all intersections of hyperplanes defined by irreducible factors of $f$, that is, the maximal building set with respect to set-theoretic inclusion. We prove the premise of Lemma~\ref{lemma:multiplicity} in this situation, namely that for every $P \in X$ there exists an $\R$-point $Q \in X$ such that
    \[
        \{i \mid P \in E_i\} = \{i \mid Q \in E_i\}.
    \] 
    This also implies the premise of Lemma~\ref{lemma:lambda}. 

    Let $P \in X$ and, without loss of generality, let $E_1,\ldots,E_r$ be the prime divisors of $\div(f \circ \rho)$ that contain $P$. Now, Lemma~\ref{lemma:intersection-in-blowup} implies that the $E_i$ are prime divisors induced by an inclusion chain of elements in $\mathcal{L}$, say $E_i = D_{U_i}$, where $U_i \in \mathcal{L}$ with $U_1 \subseteq U_2 \subseteq \ldots \subseteq U_r$. As $P \notin E_i$ for $i \geq r+1$, we may consider the problem on the open subset $ \Omega = X \setminus \bigcup_{i = r+1}^m E_i$ of $X$. As a blow-up is an isomorphism away from its center, the closed subvariety $\Omega \cap \bigcap_{i = 1}^r E_i \subseteq \Omega$ is birationally equivalent over $\R$ to the pullback $V$ of $U_r$ under the map $\sigma: Y \to \A^d_\C$ that is the chain of blow-ups along $U_1,\ldots,U_r$, in this order. In particular, $\sigma$ is defined over $\R$ as it is a sequence  of blow-ups whose centers are real linear spaces, and it suffices to prove the premise of Lemma~\ref{lemma:multiplicity} for the point in $V$ corresponding to $P$. We show the stronger statement that $V = V_1 \cup \ldots \cup V_r$ is actually an arrangement of hyperplanes, where $V_i$ is the prime divisor of $V$ corresponding to $U_i$. So, the $V_i$ are defined over $\R$ and, hence, $V_1 \cap \ldots \cap V_r$ is a linear variety defined over $\R$ and, therefore, has dense real points.

    To prove that $V$ can be defined by a product of linear forms, note that, if $B \subseteq \A^d_\C$ is any linear subspace that is a union of lines perpendicular to $U_i$ then, by definition, the proper transform of $B$ under the blow-up along $U_i$ is linear and the exceptional divisor is a union of lines perpendicular to this proper transform. This inductively shows that every $V_i$ is linear as follows: The basis of the induction is that $U_1,\ldots,U_r$ are linear spaces. For the induction step, suppose that after the blow-up along $U_i$, all the proper transforms and exceptional divisors corresponding to $U_1,\ldots,U_r$ are still linear. Moreover, we know by the above discussion that the exceptional divisors coming from $U_1,\ldots,U_i$ consist of lines perpendicular to the proper transform of $U_{i+1}$. Furthermore, as the proper transform of $U_{i+1}$ is contained in the proper transforms of $U_{i+2},\ldots,U_r$, the latter also consist of lines perpendicular to the former. So, blowing up along the proper transform of $U_{i+1}$ again yields a central hyperplane arrangement which, up to a linear translation, is defined by a product of linear forms. The process terminates at $i = r$ and the statement follows. 
\end{proof}

The goal of this paper was to derive formulas for the real log canonical threshold and its multiplicity for a singularity resembling a (not necessarily reduced) union of hyperplanes, based solely on the combinatorial data of the hyperplane arrangement. This is achieved in Corollary~\ref{corollary:rlct}. It is worth noting that the formulas can be implemented in any computer program capable of performing linear algebra calculations, see Section~\ref{section:code}.

\begin{corollary}\label{corollary:rlct}
    Let $f = L_1^{s_1} \cdots L_n^{s_n}$, where each $L_i \in \R[x_1,\ldots,x_d]$ is a linear form and the $s_i$ are positive integers. Let $\mathcal{G}$ be a building set of the (set-theoretic) hyperplane arrangement $\mathcal{A}$ defined by $f$ and let $\rlct(f) = (\lambda_\R,m_\R)$ be the real log canonical threshold of $f$ and its multiplicity. Denote by $L(\mathcal{A})$ the set of all intersections ($\neq \R^d$) of elements of $\mathcal{A}$. Then 
    \[
    \lambda_\R = \min_{W \in \mathcal{G}} \ddfrac{\text{codim}(W)}{s(W)} = \min_{W \in L(\mathcal{A})} \ddfrac{\text{codim}(W)}{s(W)}
    \]
    and
    \[
    m_\R = \max_{{ \substack{W_0\subseteq \ldots \subseteq W_{d-1}\\ W_i \in L(\mathcal{A}) } }} \bigg\vert \Set{  i \in \{0,\ldots,d-1\} | \lambda_\R = \frac{\text{codim}(W_i)}{s(W_i)}} \bigg\vert,
    \]
    where, for $W \in L(\mathcal{A})$,
    \[
    s(W) = \sum_{{\substack{j = 1 \\ W \subseteq H_j}}}^n s_j.
    \]
\end{corollary}

\begin{proof}
   This follows immediately from Theorem~\ref{theorem:lct} and Theorem~\ref{theorem:lct=rlct_for_hyperplane}, where the second formula for $\lambda$ is the direct application of Theorem~\ref{theorem:lct} to the building set $L(\mathcal{A})$.
\end{proof}

We end this section with a remark that justifies the restriction to central hyperplane arrangements and shows that this is no loss of generality.

\begin{remark}\label{remark:local-rlct}
It is shown in~\cite[Lemma 7.2]{tubes} how $\rlct(f)$ can be described in a local manner. Namely, for an Euclidean open subset $\Omega \subseteq \A^d_\R$ the authors define $\rlct_\Omega(f)$ and show that, for every $x \in \A^d_\R$ there exists an open neighbourhood $\Omega_x$ of $x$ such that $\rlct_{\Omega_x}(f) = \rlct_\Omega(f)$ for every open neighbourhood $\Omega \subseteq \Omega_x$. One can then set $\rlct_x(f) = \rlct_{\Omega_x}(f)$ to be the \textit{local real log canonical threshold} of $f$ at $x$.

    For real numbers $\lambda,\lambda',m,m'$, we define
\[
(\lambda,m) < (\lambda',m') \text{ if and only if } (\lambda < \lambda' \text{ or } (\lambda = \lambda' \text{ and } m > m') ).
\]
With this order, \cite[Lemma 7.2]{tubes} shows that $\rlct(f) = \min_{x \in \A^d_\R} \rlct_x(f)$ for a polynomial $f \in \R[x_1,\ldots,x_d]$.

If $f$ defines a hyperplane arrangement, this statement can even be interpreted without any further knowledge of the definition of $\rlct_x(f)$. Namely, for $x \in \A^d_\R$, let $f_x \in \R[x_1,\ldots,x_d]$ be the product of all irreducible factors of $f$ that vanish at $x$. In other words, $f_x$ defines the arrangement of all (not necessarily reduced) hyperplanes that appear in the hyperplane arrangement defined by $f$ and pass through $x$. In particular, $f_x$ defines a central hyperplane arrangement. In this case, \cite[Lemma 7.2]{tubes} reduces to
\[
\rlct(f) = \min_{x \in \A^d_\R} \rlct(f_x).
\]
Furthermore, this minimum can be rewritten as a minimum over just finitely many points $x_1,\ldots,x_r$ by only considering $f_x$ that define hyperplane arrangements maximal with respect to set theoretic inclusion among all the hyperplane arrangements defined by the $f_x$ and, for each of these maximal hyperplane arrangements, considering a fixed point $x_i$ that lies on all of its irreducible components.
This shows that the restriction to central hyperplane arrangements means no loss of generality.
\end{remark}

\section{Examples}\label{section:examples}

In this section, we apply our formulas for the (real) log canonical threshold and its multiplicity to illustrating examples of hyperplane singularities. Moreover, we examine several instances showing that Theorem~\ref{theorem:lct=rlct_for_hyperplane}, in general, fails but might still hold under special circumstances other than hyperplane arrangements. 
Our first example deals with line arrangements in $\R^2$ which resembles the case of two-parameter models in statistics.

\begin{example}\label{example:application-formula-two-dimensional}
     We can deduce from Corollary~\ref{corollary:rlct} a compact formula for the rlct and its multiplicity of central line arrangements in $\R^2$. Let $L_1,\ldots, L_n \in \R[x,y]$ be linear forms, $s_1 \leq \ldots \leq s_n$ positive integers, $f = L_1^{s_1} \cdots L_n^{s_n}$, and $\rlct(f)= (\lambda,m)$. Denote by $\ell_i$ the line defined by $L_i$ and by $\mathbf{0}$ the set containing the origin of $\R^2$.

     The intersection lattice is $L(\mathcal{A}) = \{\ell_i \mid i \in [n]\} \cup \{\mathbf{0}\}$. Moreover, $\codim(\ell_i) = 1$ and $\codim(\mathbf{0}) = 2$, and we get the following formulas for the multiplicities:
     \[
     s(\ell_i) = s_i \text{ and } s(\mathbf{0}) = \sum_{i = 1}^n s_i.
     \]
     Using the building set $\mathcal{G} = L(\mathcal{A})$, a direct application of Corollary~\ref{corollary:rlct} gives
     \[
     \lambda = \min \Set{ \frac{1}{s_n} , \frac{2}{\sum s_i} } \text{ and } m = \bigg\vert \Set{   \frac{1}{s_n} , \frac{2}{\sum s_i} } \bigg\vert.
     \]
     Finally,
     \begin{align*}
       \lambda = \begin{cases}
         \frac{1}{s_n} \text{ if } \sum_{i = 1}^{n-1} s_i \leq s_n \\
         \frac{2}{\sum_{i=1}^n s_i} \text{ else}
     \end{cases}  
     \end{align*}
     and 
     \begin{align*}
       m = \begin{cases}
         2 \text{ if } \sum_{i = 1}^{n-1} s_i = s_n \\
         1 \text{ else}.
     \end{cases}  
     \end{align*}

\begin{comment}
We test the formulas on two examples from before. The first one is $f(x,y) = xy$, so $s_1 = 1 = s_2$. In this case $s_1 \leq s_2$, so $\lambda = 1/s_2 = 1$. Moreover, $s_1 = s_2$ and hence $m = 2$. 

The second example is $f(x,y) = xy(x-y)(x+y)$ in which $s_1 = s_2 = s_3 = s_4 = 1$. As $s_1 + s_2 + s_3 > s_4$, we arrive at $\lambda = 2/4 = 1/2$ and $m = 1$.

We give a third example in which $\lambda < 1$ and $m = 2$. Consider $f(x,y) = xy(x-y)^2$, so $s_1 = s_2 = 1$ and $s_3 = 2$. As $s_1 + s_2 = s_3$, we have $\lambda = 1/s_3 = 1/2$ and $m = 2$.
\end{comment}
\end{example}

\begin{example}\label{example:recover-tubes}
The examples from~\cite[Example 3.1]{tubes} are easily recovered using the formulas from~\ref{example:application-formula-two-dimensional}.
\begin{itemize}
    \item[(a)]  Let $f(x,y) = x = y^0x^1$. Then $s_1 = 0$ and $s_2 = 1$. Therefore $\rlct(f) = (\lambda,m) = (1,1)$.
    \item[(b)] For $f(x,y) = xy$, we have $s_1 = s_2 = 1$ and so $\rlct(f) =  (1,2)$.
    \item[(c)] If $f(x,y) = x^2y^3$ then $s_1 = 2$ and $s_2 = 3$. Therefore, our formulas give $\rlct(f) = (1/3,1)$.
    \item[(d)] Finally, for $f(x,y) = xy(x+y)(x-y)$ we get $s_i = 1$ for $i \in \{1,2,3,4\}$. Therefore, $\lambda = 2/4 = 1/2$ and $m = 1$. 
\end{itemize}
\end{example}

\begin{example}
    Example~\ref{example:application-formula-two-dimensional} shows, in particular, that an arrangement $f$ of $n \geq 3$ distinct reduced lines in $\R^2$ satisfies $\rlct(f) = (2/n,1)$.
\end{example}

\begin{figure}[H]
    \centering
    \includegraphics[width=0.9\linewidth]{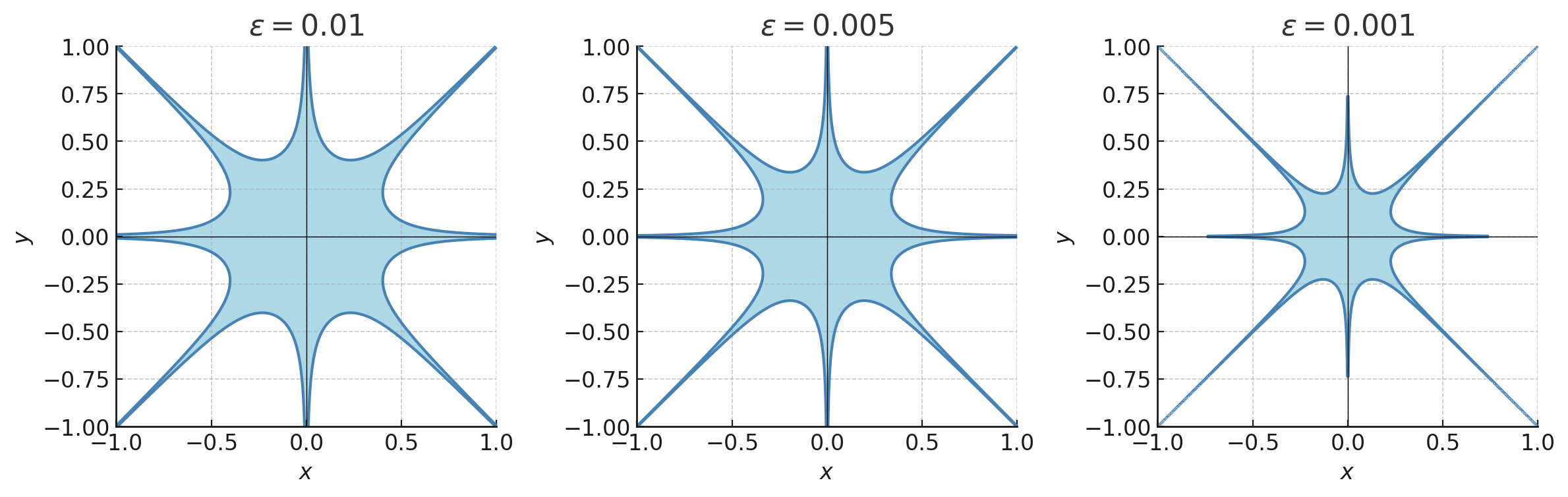}
    \caption{The area of the set of all $(x,y) \in [-1,1]^2$ with $|x^3y - xy^3| \leq \varepsilon$ goes to $0$ like $C \varepsilon^{1/2}$, for some constant $C$, because $\rlct(f) = (1/2, 1)$, see Example~\ref{example:recover-tubes}(d).}
    \label{fig:rlct}
\end{figure}

A further domain of application for the results of this paper is the asymptotic behaviour for integrals that would otherwise be hard to compute. Let $W$ be a full-dimensional compact subset of $\R^d$, $\varphi: W \to \R$ an analytic function that is strictly positive on $W$, $f: W \to \R$ an analytic function who has a zero in the interior of $W$, and let $dw$ denote the standard Lebesgue measure on $W$.  According  to~\cite[Definition 3.2]{tubes}, see also \cite[Section 4]{WatanabeNeurComp2001}, \cite[Section 7.1]{Watanabe} and \cite[Section 4.1]{Lin-dissertation}, for small \(\varepsilon > 0\), the \textit{volume function}
\begin{equation}\label{equation:integral}
V(\varepsilon) = \int_{|f(w)| \leq \varepsilon} \varphi(w) \, dw
\end{equation}
is asymptotically \( C \varepsilon^\lambda (-\ln \varepsilon)^{m-1} \) for some constant \( C \), where $\rlct_W(f) = (\lambda,m) $. When $f$ is a function whose real locus looks like a hyperplane arrangement at every interior point of $W$, we can use the formulas from Corollary~\ref{corollary:rlct} to describe the asymptotic behaviour of this integral, see Figure~\ref{fig:rlct} and Figure~\ref{fig:3D-surface}. We show how such a computation can be done for an explicit example in~$\R^3$.

\begin{example}\label{example:four-planes}
    Let $f$ be a function that defines a union of four planes in $\R^3$ that meet in a common point in the interior of the full-dimensional compact set $W \subseteq \R^3$. For instance, suppose that $W$ contains the origin and $f(x,y,z) = x y^2 z^2 (x+y+z)$.

\begin{figure}[H]
    \centering
    \includegraphics[width=0.6\linewidth]{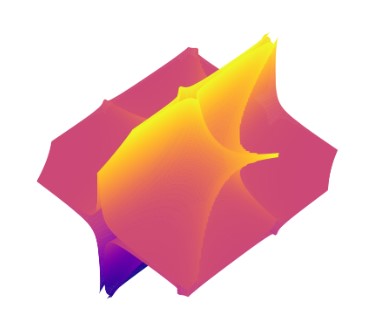}
    \caption{The surface defined by $|xy^2z^2 (x+y+z)| = \varepsilon$ in a compact neighbourhood of the origin in $\R^3$, for small $\varepsilon >0$. It consists of sixteen connected components. As $\rlct(xy^2z^2 (x+y+z)) = (1/2,3)$, see Example~\ref{example:four-planes}, the area between these components asymptotically equals $C\varepsilon^{1/2}(-\ln \varepsilon)^2$, for some constant~$C$.}
    \label{fig:3D-surface}
\end{figure}

    The intersection lattice of the hyperplane arrangement defined by $f$ contains three types objects: planes~$\mathcal{H}$, lines $\mathcal{\ell}$ which are the intersection of two of the planes and are not contained in any of the other planes, and the origin $\mathbf{0}$. The possible values of $\codim(\mathcal{H})/s(\mathcal{H})$ for the four planes are
    \[
    1 \text{ and } \frac{1}{2},
    \]
    where the values $1$ come from the planes with multiplicity $2$, that is $y^2 = 0$ and $z^2 = 0$. For the origin, we have 
    \[
    \frac{\codim(\mathbf{0})}{s(\mathbf{0})} = \frac{3}{6} = \frac{1}{2}.
    \]
    Finally, the values of $\codim(\ell)/s(\ell)$, as $\ell$ varies through all the lines, are
    \[
    \frac{2}{4} = \frac{1}{2}, \frac{2}{3}, \text{ and } \frac{2}{2}= 1.
    \]
    The minimum of all these values is the real log canonical threshold $\lambda_\R = 1/2$, by Corollary~\ref{corollary:rlct}. Among the lines, the intersection of the two planes of multiplicity $2$ attains the minimum value of $1/2$. Additionally, each of these two planes and the origin attain this value. So, we have a chain $\mathbf{0} \subseteq \ell \subseteq \mathcal{H}$ of (three) elements $W$ in the intersection lattice with $\codim(W)/s(W) = \lambda_\R$. This is the maximal possible length of such a sequence and hence $m_\R = 3$. In total,
    \[
    \rlct(f) = (1/2,3)
    \]
    and so, for small $\varepsilon > 0$, the integral of Equation (\ref{equation:integral}) is asymptotically $C\varepsilon^{1/2}(-\ln \varepsilon)^2$, for some constant $C$, see Figure~\ref{fig:3D-surface}.
\end{example}

The following is an easy example in two-dimensional affine space where $\lambda_\C \neq \lambda_\R$ and $m_\C \neq m_\R$.

\begin{example}
    Let $f(x,y) = xy(y^2+1)^2$ whose real locus is the union of two lines. As the factor $y^2 +1$ has empty real locus, we can use~\cite[Theorem 7.1]{tubes} which shows that the real log canonical threshold is given by $\rlct(f) = (1,2)$. 
    
    Over $\C$, the function factors as $xy(y - \sqrt{-1})^2(y+\sqrt{-1})^2$ which is locally the union of (at most) two not necessarily reduced lines. We compute the $\lct$ locally at each point and take the minimum which is possible by~\ref{remark:local-rlct}. At points $w$ just lying on one of the lines, we either get $\lct_w(f) = (1,1)$ or $(1/2,1)$ depending on whether the line is reduced or not. In the point $w = (0,0)$, the pair $\lct$ is $(1,2)$, just as in the real case. Finally, if $w = (0,\pm \sqrt{-1})$ then again $\lct_w(f) = (1/2,1)$. Summing this up by taking the minimum of all these pairs with respect to the total order from~\ref{remark:local-rlct}, we infer that $\lct(f) = (1/2,1)$.
    
\end{example}

\begin{example}\label{example:3D-complex-real-equal}
    We consider the function $f(x,y,z) = x^2 + y^2 + z^2$. It defines an irreducible surface in $\C^3$. Its real locus is just an isolated reduced point which cannot lie Zariski dense in the surface. We will show that still $\rlct(f) = \lct(f)$.
    To compute its log canonical threshold over $\R$ and over $\C$, we use a blow up at the origin which can be covered by three open affine charts who are symmetric with respect to the three coordinates $x,y,z$. Therefore, it suffices to consider the blow up just on one chart, for instance,
    \begin{align*}
        \rho: \C^3 &\to \C^3\\
                (x,y,z) &\mapsto (x,xy,xz).
    \end{align*}
    The Jacobian determinant of $\rho$ is $\det \text{Jac}(\rho)(x,y,z) = x^2$ and the pullback of $f$ is $(f \circ \rho)(x,y,z) = x^2 (1 + y^2 + z^2)$. As $1+y^2+z^2$ has no zeros over $\R$, we can use~\cite[Theorem 7.1]{tubes} and infer that
    \[
    \rlct(f) = (3/2,1).
    \]
    The factor $1+y^2+z^2$ defines a non-singular irreducible surface in three-dimensional complex affine space. Its tangent space, which is the kernel of its Jacobian matrix $[0 \ 2y \ 2z]$, contains first standard basis vector of $\C^3$ for arbitrary choices of $y$ and $z$. But this vector is not contained in the surface defined by the first factor $x^2$ of $f \circ \rho$. So, reminding ourselves that $\det \text{Jac}(\rho)(x,y,z) = x^2$ and using that the complex scheme defined by $f\circ \rho$ is locally isomorphic to the one defined by $x^2 y$, we infer that also
    \[
    \lct(f) = (3/2,1).
    \]
    
\end{example}

\section{SageMath code for the RLCT of a hyperplane arrangement}\label{section:code}

The following is a SageMath function~\cite{sagemath} that takes as an input an $(n \times d)$-matrix \texttt{A} whose rows are the vectors of (complex) coefficients of linear forms $L_1,\ldots,L_n$ in $d$ variables and a row vector \texttt{s} whose entries are positive integers $s_1,\ldots,s_n$. The output is the pair $\lct(f) = (\lambda_\C,m_\C)$ of the function $f = L_1^{s_1} \cdots L_n^{s_n}$. This coincides with the pair $\rlct(f) = (\lambda,m)$ in case that all the entries of \texttt{A} are real numbers. 
\newpage
\begin{lstlisting}[language=Python, caption={SageMath code example}]
def compute_lct(A, s):
    from itertools import combinations
    
    # Check if the dimension of s equals the number of rows of A
    if len(s) != A.nrows():
        raise ValueError("The dimension of s should equal the number of rows of A.")

    d = A.ncols()  # Dimension of the ambient space
    n = A.nrows()  # Number of hyperplanes
    def codim(W):
        return d - W.dimension()
    def s_function(W):
        # Compute sum of s_j for all hyperplanes H_j containing W
        return sum(s[j] for j in range(n)
                   if all(vector(A[j, :]).dot_product(v) == 0
                          for v in W.basis()))

    # Calculate all intersections
    intersections = []
    for r in reversed(range(1, n + 1)):
        for comb in combinations(range(n), r):
            submatrix = A[list(comb), :]
            intersection = submatrix.right_kernel()
            intersections.append(intersection)
    
    # Remove duplicate intersections
    L = list(uniq(intersections))
    
    # Compute lambda as the minimum of codim(W) / s_function(W)
    lambda_min = min(codim(W) / s_function(W) for W in L)
    
    # Filter intersections where lambda achieves its minimum
    min_lambda_intersections = [W for W in L
                                if codim(W) / s_function(W) == lambda_min]
    
    # Calculate the maximum length of a chain of intersections
    def longest_chain_length(L):
        if not L:
            return 0
        
        lengths = [1] * len(L)
        
        for i in range(len(L)):
            for j in range(i):
                if all(v in L[i] for v in L[j].basis()):
                    lengths[i] = max(lengths[i], lengths[j] + 1)
        
        return max(lengths)
    m = longest_chain_length(min_lambda_intersections)
    return lambda_min, m
\end{lstlisting}

\vspace{1cm}

The code is also available on \href{https://github.com/DanielWindisch/compute-lct}{https://github.com/DanielWindisch/compute-lct}. 
To use SageMath online, visit \href{https://sagecell.sagemath.org/}{SageMathCell}.
We can test the SageMath command on the function from Example~\ref{example:four-planes} as described below.

\vspace{0.5cm}
\begin{verbatim}
A = Matrix([[1,0,0],[0,1,0],[0,0,1],[1,1,1]])
s = vector([1,2,2,1])

lambda_value, m_value = compute_lct(A, s)
lambda_value, m_value
\end{verbatim}
\vspace{0.5cm}

 It gives back the correct pair $(1/2,3)$ within an average time of 24.4 milliseconds on a hp Probook 450 G6 with intel Core i5-8265U @ 1.6GHz, 4 cores / 8 threads, and Windows 11. The computation of just the log canonical threshold with the command \texttt{lct} in the \textit{Macaulay2} package~\cite{M2} BernsteinSato takes 4.67 seconds in the same setup.
 Moreover, in the same setup, tests with random choices of coefficients and exponents in the linear forms defining the hyperplane arrangements show that our code is able to perform computations of the rlct for arrangements of up to $n = 13$ hyperplanes in an affine space of dimension at most $d = 9$ within two minutes. Reducing to $d = 3$, we can raise the number of hyperplanes to $n = 15$.

 Our code can be easily extended to hyperplane arrangements that are not necessarily central, see Remark~\ref{remark:local-rlct}. The computations of the rlct of the central parts of a given hyperplane arrangement could even be parallelized.

\section*{Acknowledgements}
We would like to thank Bernd Sturmfels for suggesting the problem of determining the rlct of hyperplane arrangements to us. We are also grateful to Lorenzo Baldi, Daniel L. Bath, and Zach Teitler for helpful discussions on the topic. D. Kosta gratefully acknowledges funding from the Royal Society Dorothy Hodgkin Research Fellowship DHF$\backslash$R1$\backslash$201246. We would like to thank the associated Royal Society Enhancement grant RF$\backslash$ERE$\backslash$231053  that supported  D. Windisch's postdoctoral research.  A great thanks goes to the two anonymous referees whose comments and suggestions substantially improved the manuscript.

\bibliographystyle{amsplainurl}
\bibliography{bibliography}
 
\vspace{1cm} 
 
\noindent
\textsc{Dimitra Kosta, School of Mathematics, University of Edinburgh, James Clerk Maxwell Building, Peter Guthrie Tait Road, EH9 3FD,  Edinburgh, United Kingdom} \\
\textit{E-mail address}: \texttt{D.Kosta@ed.ac.uk}\\ 
 
\noindent
\textsc{Daniel Windisch, Department of Computer Science, KU Leuven, Celestijnenlaan 120A, Leuven, Belgium} \\
\textit{E-mail address}: \texttt{daniel.windisch.math@gmail.com}

\end{document}